\documentclass[11pt]{article}
\usepackage{amsmath, amssymb, amsthm}
\usepackage{enumitem}
\usepackage{hyperref}
\usepackage{cleveref}

\newtheorem{theorem}{Theorem}[section]
\newtheorem{lemma}[theorem]{Lemma}
\newtheorem{proposition}[theorem]{Proposition}
\newtheorem{corollary}[theorem]{Corollary}
\newtheorem{definition}[theorem]{Definition}
\newtheorem{remark}[theorem]{Remark}
\newtheorem{example}[theorem]{Example}

\title{A $B$-Restricted Clique Polynomial and Connections to Tanner's Inequality}
\author{Hossein Teimoori Faal}
\date{Feb 21, 2026}

\begin{document}

\maketitle


\begin{abstract}
	Let $G$ be a finite simple graph and $B \subseteq V(G)$.  
	We study the \emph{$B$-restricted clique polynomial} $C_B(G;x)$, including its weighted version allowing vertex multiplicities, as a versatile tool to capture structural properties of vertex subsets.
	
	First, we develop a complete deletion theory for $C_B(G;x)$, including vertex and edge recurrences that generalize classical clique polynomial results.  
	These recurrences yield monotonicity principles for the largest negative root $\zeta_G(B)$: it is monotone under induced subgraphs and reverse-monotone under spanning subgraphs.  
	Consequently, we derive explicit bounds on $B$-independence numbers, chromatic numbers, $B$-girth, and Hamiltonicity constraints, showing that $\zeta_G(B)$ serves as a unifying local invariant.
	
	Next, we connect $B$-clique polynomials to spectral graph theory.  
	For $(n,d,\lambda)$-graphs, spectral techniques, including the Expander Mixing Lemma and Tanner's inequality, provide uniform bounds on $B$-restricted clique coefficients, demonstrating that clique growth within $B$ is naturally controlled by the spectral gap.
	
	Finally, we show that weighted $B$-clique polynomials encode \emph{homomorphism constraints}.  
	Specifically, if $f: G \to H$ is a surjective homomorphism mapping $B_G$ onto $B_H$, then $\zeta_G(B_G) \ge \zeta_H(B_H)$, yielding a local \emph{no-homomorphism criterion} based on $B$-roots.
	
	Overall, $C_B(G;x)$ provides a unified framework capturing combinatorial, spectral, and homomorphic information in vertex-restricted analysis, highlighting its power for both global and local structural insights.
\end{abstract}


\section{Introduction}

Clique polynomials encode fundamental combinatorial information about graphs.  
For a finite simple graph $G=(V,E)$, the \emph{clique polynomial} is
\[
C(G;x) = \sum_{i\ge 0} c_i x^i,
\]
where $c_i$ counts the $i$-vertex cliques in $G$.  
These polynomials have been extensively studied, with connections to dependence polynomials \cite{fisher1990dependence}, root location \cite{hajiabolhassan1998clique}, and extremal graph theory \cite{hoede1994clique}.  

\subsection*{$B$-Restricted Clique Polynomials}

Given a subset $B\subseteq V(G)$, we define the \emph{$B$-restricted clique polynomial}
\[
C_B(G;x) := \sum_{i=0}^{\omega(G[B])} c_i(B) x^i,
\quad c_i(B) = |\{K\subseteq B : |K|=i \text{ and } K \text{ is a clique in } G\}|,
\]
which explicitly emphasizes the role of $B$ within the ambient graph $G$.  
While formally equivalent to $C(G[B];x)$, this viewpoint allows $G$ to remain fixed and $B$ to vary, revealing new algebraic, extremal, and structural phenomena.  

We also consider \emph{weighted $B$-clique polynomials}, where each $v_i \in B$ carries a positive integer weight $w_i$.  
Weighted polynomials count cliques according to vertex multiplicities, corresponding combinatorially to \emph{$B$-blow-up graphs}.  
This extension naturally connects to graph homomorphisms and provides refined tools for combinatorial analysis.

\subsection*{Homomorphism Perspective}

Let $f: G \to H$ be a surjective graph homomorphism.  
For each $x \in V(H)$, define $A_x = f^{-1}(x)\subseteq V(G)$.  
These sets are independent and partition $V(G)$.  
In the $B$-restricted setting, for $B_G \subseteq V(G)$, the restriction $f(B_G) \subseteq V(H)$ induces a \emph{$B$-partition} of $B_G$ into independent subsets.  
Weighted $B$-clique polynomials then provide a combinatorial tool to detect local homomorphism obstructions, giving a \emph{no-homomorphism criterion}.

\subsection*{Our Contributions}

This paper develops a systematic study of $B$-restricted clique polynomials and their largest negative root
\(\zeta_G(B)\), highlighting combinatorial, spectral, and homomorphic applications.

\begin{enumerate}
	\item \textbf{Structural Recurrences.} Vertex and edge deletion formulas for $C_B(G;x)$ mirror classical clique polynomial recurrences \cite{hajiabolhassan1998clique}, providing an algebraic backbone for analysis.  
	
	\item \textbf{Root Monotonicity and Extremal Applications.}  
	The largest negative root $\zeta_G(B)$ is monotone under induced subgraphs and reverse-monotone under spanning subgraphs.  
	Consequently, we derive tight bounds on $B$-independence numbers, $B$-chromatic numbers, $B$-girth, and Hamiltonicity within $B$, showing that $\zeta_G(B)$ serves as a precise local invariant controlling combinatorial structures \cite{fisher1990dependence,hoede1994clique}.  
	
	\item \textbf{Spectral Constraints.}  
	For $(n,d,\lambda)$-graphs, spectral techniques and Tanner's isoperimetric inequality \cite{tanner1984explicit,spiesser2010spectral} yield bounds on $B$-restricted clique coefficients.  
	This links $\zeta_G(B)$ to spectral pseudorandomness and expansion, showing that clique growth inside $B$ is governed by the spectral gap.  
	
	\item \textbf{Weighted and Homomorphism Applications.}  
	Weighted $B$-clique polynomials encode combinatorial homomorphism constraints: if $f: G \to H$ is surjective and respects $B$, then
	\(\zeta_G(B) \ge \zeta_H(B)\).  
	This establishes a local no-homomorphism criterion and demonstrates that $B$-clique polynomials unify combinatorial, spectral, and homomorphic information.
\end{enumerate}

\subsection*{Organization of the Paper}

The remainder of the paper is organized as follows.  
In Section~2, we introduce notation, review classical clique polynomials, and define the $B$-restricted and weighted clique polynomials.  
Section~3 develops structural recurrences and establishes monotonicity properties of the largest negative root $\zeta_G(B)$.  
Section~4 presents extremal applications, deriving $B$-independence, $B$-chromatic, $B$-girth, and Hamiltonicity bounds controlled by $\zeta_G(B)$.  
In Section~5, we connect $B$-clique polynomials to spectral graph theory, using the Expander Mixing Lemma and Tanner's isoperimetric inequality to bound clique growth within $B$.  
Section~6 explores weighted $B$-clique polynomials and their application to graph homomorphisms, establishing a local no-homomorphism criterion.  
Finally, Section~7 concludes with a discussion of open problems and directions for future research.


\subsection{Local Restriction Notation}

For $v\in V$ and $B\subseteq V$, define
\[
N_B(v):=N(v)\cap B.
\]
More generally, for a clique $K\subseteq B$, define
\[
N_B(K):=\bigcap_{v\in K} N_B(v).
\]
Thus $N_B(K)$ consists precisely of those vertices in $B$
that can extend $K$ to a larger clique.

This notation will allow us to formulate deletion recurrences
in a form directly parallel to the classical theory.

\subsection{Roots and Extremal Parameters}

Since $C_B(G;0)=1$ and all coefficients are nonnegative,
every real root of $C_B(G;x)$ is negative.

\begin{definition}
	Let
	\[
	Z_B(G)
	=
	\{x<0 : C_B(G;x)=0\}.
	\]
	If $Z_B(G)\neq\emptyset$, define
	\[
	\zeta_G(B)=\max Z_B(G),
	\]
	the largest negative root of $C_B(G;x)$.
	If $Z_B(G)=\emptyset$, set $\zeta_G(B)=-\infty$.
\end{definition}

We also define the following extremal parameters restricted to $B$:

\begin{itemize}
	\item The \emph{$B$-independence number}
	\[
	\alpha_B(G)
	=
	\max\{|I| : I\subseteq B \text{ independent in } G\}.
	\]
	
	\item The clique number of $G[B]$, denoted $\omega(G[B])$.
\end{itemize}

As in the classical theory, the parameter $\zeta_G(B)$ will serve as
a bridge between algebraic information (roots of the polynomial)
and combinatorial structure (independence and clique bounds).

A fundamental tool in this context is the \emph{Expander Mixing Lemma} \cite{spiesser2010spectral},
which asserts that for any subsets $X,Y\subseteq V(G)$,
\[
\left|
e(X,Y)-\frac{d|X||Y|}{n}
\right|
\le
\lambda\sqrt{|X||Y|},
\]
where $e(X,Y)$ denotes the number of edges with one endpoint in $X$ and the other in $Y$.

\section{Recurrence Relations and Root Monotonicity}

In this section we develop the structural recurrences for the
$B$-restricted clique polynomial and derive the fundamental
monotonicity properties of its largest negative root.
These results parallel the classical theory for ordinary
clique polynomials and form the algebraic backbone of the paper.

\subsection{Vertex Deletion}

\begin{lemma}[Vertex Deletion Recurrence]\label{lem:vertex-del}
	Let $v\in V(G)$ and $B\subseteq V(G)$. Then
	
	\begin{enumerate}
		\item[(i)] If $v\in B$, then
		\[
		C_B(G;x)
		=
		C_{B\setminus\{v\}}(G-v;x)
		+
		x\,C_{B\cap N(v)}(G[N(v)];x).
		\]
		
		\item[(ii)] If $v\notin B$, then
		\[
		C_B(G;x)=C_B(G-v;x).
		\]
	\end{enumerate}
\end{lemma}

\begin{proof}
	Every clique $K\subseteq B$ either contains $v$ or does not.
	
	If $v\notin K$, then $K$ is a clique of $G-v$
	contained in $B\setminus\{v\}$.
	These are counted by $C_{B\setminus\{v\}}(G-v;x)$.
	
	If $v\in K$, then $K=\{v\}\cup K'$,
	where $K'$ is a clique in $G[N(v)]$
	contained in $B\cap N(v)$.
	Each such $K'$ contributes one higher degree term,
	giving the factor $x$.
	
	If $v\notin B$, no clique in $B$ can contain $v$,
	so deletion of $v$ does not affect $C_B(G;x)$.
\end{proof}

\begin{example}
	Let $G = K_3$ (triangle) with vertices $\{a,b,c\}$, and let $B = \{a,b\}$. Then:
	\begin{itemize}
		\item $G[B]$ is an edge $ab$.
		\item $c_0(B) = 1$, $c_1(B) = 2$ (vertices $a$ and $b$), $c_2(B) = 1$ (edge $ab$).
		\item Thus $C_B(G;x) = 1 + 2x + x^2 = (1+x)^2$.
	\end{itemize}
	Check Lemma \ref{lem:vertex-del} with $v = a \in B$:
	\begin{itemize}
		\item $C_{B\setminus\{a\}}(G\setminus a;x)$: $B\setminus\{a\} = \{b\}$, $G\setminus a$ is an isolated vertex $b$, so polynomial is $1 + x$.
		\item $C_{B\cap N(a)}(G[N(a)];x)$: $N(a) = \{b,c\}$, $B\cap N(a) = \{b\}$, $G[N(a)]$ is edge $bc$, restricted to $\{b\}$ gives polynomial $1 + x$.
		\item $x \cdot (1+x) = x + x^2$.
		\item Sum: $(1+x) + (x + x^2) = 1 + 2x + x^2$, matching $C_B(G;x)$.
	\end{itemize}
\end{example}


\subsection{Edge Deletion}

\begin{lemma}[Edge Deletion Recurrence]\label{lem:edge-del}
	Let $uv\in E(G)$ and $B\subseteq V(G)$. Then
	\[
	C_B(G;x)
	=
	C_B(G-uv;x)
	+
	x^2\,C_{B\cap N(u)\cap N(v)}(G[N(u)\cap N(v)];x).
	\]
\end{lemma}

\begin{proof}
	A clique $K\subseteq B$ either does not contain both $u$ and $v$,
	or it does.
	
	If it does not, then it remains a clique in $G-uv$.
	If it contains both, then
	$K=\{u,v\}\cup K'$,
	where $K'$ is a clique in $G[N(u)\cap N(v)]$
	contained in $B\cap N(u)\cap N(v)$.
	Each such extension increases the degree by $2$,
	yielding the stated identity.
\end{proof}

\subsection{Existence of a Real Root}

We now show that $C_B(G;x)$ always has a real root in $[-1,0)$
whenever $B$ is nonempty.

\begin{theorem}[Existence of a Negative Root]\label{thm:existence}
	Let $B\subseteq V(G)$ be nonempty.
	Then $C_B(G;x)$ has at least one real root in the interval $[-1,0)$.
\end{theorem}

\begin{proof}
	Choose any vertex $u\in B$ and consider the induced subgraph
	$H=G[\{u\}]$.
	Then
	\[
	C_B(H;x)=1+x,
	\]
	whose unique real root is $-1$.
	
	We will prove below (Theorem~\ref{thm:induced-mon})
	that the largest negative root is monotone under induced subgraphs.
	Applying that result gives
	\[
	-1=\zeta_H(B)\le \zeta_G(B).
	\]
	
	Since $C_B(G;0)=1>0$ and all coefficients are nonnegative,
	any real root must lie in $(-\infty,0)$.
	Thus $\zeta_G(B)\in[-1,0)$,
	and in particular a real root exists.
\end{proof}

\subsection{Monotonicity Under Induced Subgraphs}

\begin{theorem}[Induced Subgraph Monotonicity]
	\label{thm:induced-mon}
	Let $H$ be an induced subgraph of $G$
	and let $B\subseteq V(H)$.
	Then
	\[
	\zeta_H(B)\le \zeta_G(B).
	\]
\end{theorem}

\begin{proof}
	It suffices to consider $H=G-v$.
	Let $\zeta=\zeta_{G-v}(B)$.
	
	If $v\notin B$, then
	$C_B(G;x)=C_B(G-v;x)$
	and the roots coincide.
	
	If $v\in B$, then by Lemma~\ref{lem:vertex-del},
	\[
	C_B(G;x)
	=
	C_{B\setminus\{v\}}(G-v;x)
	+
	x\,C_{B\cap N(v)}(G[N(v)];x).
	\]
	
	Substitute $x=\zeta$.
	Since $\zeta$ is a root of $C_B(G-v;x)$,
	the first term vanishes, giving
	\[
	C_B(G;\zeta)
	=
	\zeta\,
	C_{B\cap N(v)}(G[N(v)];\zeta).
	\]
	
	By induction on $|V(G)|$,
	the largest negative root of the induced graph
	$G[N(v)]$ is at most $\zeta_G(B)$.
	Hence the second factor is nonnegative.
	Since $\zeta<0$,
	we obtain $C_B(G;\zeta)\le0$.
	
	Because $C_B(G;0)=1>0$,
	the Intermediate Value Theorem
	implies that $C_B(G;x)$ has a root in $[\zeta,0)$.
	Thus $\zeta_G(B)\ge\zeta$.
\end{proof}

\subsection{Monotonicity Under Spanning Subgraphs}

\begin{theorem}[Spanning Subgraph Monotonicity]
	\label{thm:spanning-mon}
	Let $H$ be a spanning subgraph of $G$.
	Then for any $B\subseteq V(G)$,
	\[
	\zeta_G(B)\le \zeta_H(B).
	\]
\end{theorem}

\begin{proof}
	It suffices to consider $H=G-uv$.
	Using Lemma~\ref{lem:edge-del} and substituting
	$x=\zeta_G(B)$,
	one obtains
	\[
	C_B(G-uv;\zeta_G(B))
	=
	-\zeta_G(B)^2
	\,C_{B\cap N(u)\cap N(v)}(G[N(u)\cap N(v)];
	\zeta_G(B)).
	\]
	
	By induced monotonicity,
	the second factor is nonnegative.
	Since $-\zeta_G(B)^2<0$,
	the right-hand side is nonpositive.
	As $C_B(G-uv;0)=1>0$,
	there must be a root in $[\zeta_G(B),0)$.
	Hence $\zeta_{G-uv}(B)\ge\zeta_G(B)$.
\end{proof}

\subsection{Root Interval}

Combining the previous results gives a sharp universal interval.

\begin{corollary}\label{cor:interval}
	For every nonempty $B\subseteq V(G)$,
	\[
	-1\le \zeta_G(B)<0.
	\]
\end{corollary}

\begin{proof}
	The lower bound follows from Theorem~\ref{thm:existence}.
	The strict upper bound follows from $C_B(G;0)=1>0$.
\end{proof}

\begin{remark}
	This corollary guarantees that every $B$-restricted clique polynomial has at least one real root. 
	Note that this property is not generally true for the face polynomial of a simplicial complex, 
	where real roots may not exist. Thus, clique polynomials enjoy a robust real-rootedness property 
	that is absent in more general combinatorial generating polynomials.
\end{remark}

\subsection{Clique Polynomials and Graph Operations}

In this subsection we record two basic structural properties of clique polynomials 
under standard graph operations. These identities will be used repeatedly in the sequel.

\begin{definition}
	The \emph{disjoint union} of two graphs $G$ and $H$, denoted by $G \cup H$, 
	is the graph with vertex set $V(G)\cup V(H)$ and edge set $E(G)\cup E(H)$, 
	where $V(G)$ and $V(H)$ are assumed disjoint.
\end{definition}

\begin{proposition}
	For any two graphs $G$ and $H$, we have
	\[
	C(G \cup H,x)=C(G,x)+C(H,x)-1.
	\]
\end{proposition}

\begin{proof}
	Every clique in $G \cup H$ is either a clique of $G$ or a clique of $H$. 
	The empty clique is counted twice in $C(G,x)+C(H,x)$; hence we subtract $1$.
\end{proof}

\begin{definition}
	The \emph{join} of two graphs $G$ and $H$, written $G \vee H$, 
	is obtained from $G \cup H$ by adding all edges between 
	$V(G)$ and $V(H)$.
\end{definition}

\begin{proposition}
	For any graphs $G$ and $H$, we have
	\[
	C(G \vee H,x)=C(G,x)\,C(H,x).
	\]
\end{proposition}

\begin{proof}
	Every clique in $G \vee H$ is obtained uniquely as the join 
	of a clique of $G$ and a clique of $H$. 
	Hence the generating functions multiply.
\end{proof}



\section{Connection to Tanner's Isoperimetric Inequality and Spectral Bounds}

\label{sec:spectral-tanner}

In this section, we connect the $B$-restricted clique polynomial to spectral graph theory.  
In particular, we show how combinatorial bounds on $\zeta_G(B)$ and spectral constraints on subsets $B$ interact to provide a unified framework controlling clique growth.

\subsection{Spectral Preliminaries}

Recall that an \emph{$(n,d,\lambda)$-graph} is a $d$-regular graph on $n$ vertices whose adjacency matrix $A$ has eigenvalues
\[
\lambda_1=d \ge \lambda_2 \ge \dots \ge \lambda_n \ge -d, 
\qquad \text{with } |\lambda_i|\le \lambda \text{ for } i\ge 2.
\]
Such graphs are fundamental in spectral graph theory \cite{spiesser2010spectral}. 

The \emph{Expander Mixing Lemma (EML)} states that for any subsets $X,Y \subseteq V(G)$,
\begin{equation}\label{eq:EML}
\left| e(X,Y) - \frac{d|X||Y|}{n} \right| \le \lambda \sqrt{|X||Y|}.
\end{equation}
Here, $e(X,Y)$ denotes the number of edges with one endpoint in $X$ and the other in $Y$.  
The EML ensures that for a spectral expander, the edge distribution between any two subsets behaves approximately like a random graph with edge density $d/n$.  
This inequality is critical for controlling $|N_B(v)|$ and, more generally, $|N_B(K)|$ for cliques $K \subseteq B$, which in turn bounds $c_i(B)$, the number of $i$-cliques in $B$.

\subsection{Tanner's Isoperimetric Inequality}

Tanner \cite{tanner1984explicit} provided an explicit \emph{isoperimetric inequality} for $(n,d,\lambda)$-graphs:

\begin{theorem}[Tanner's Inequality \cite{tanner1984explicit}]
	Let $G$ be an $(n,d,\lambda)$-graph, and let $S \subseteq V(G)$ be any subset of size $|S| = bn$, $0 < b \le 1$. Then
	\[
	|N(S)| \ge n \left( 1 - \frac{\lambda^2}{d^2} \cdot \frac{1-b}{b} \right),
	\]
	provided the right-hand side is positive.
\end{theorem}

This inequality gives a \emph{lower bound} on the neighborhood size of any subset of vertices.  
In our $B$-restricted setting, let $B \subseteq V(G)$ and $|B| = m$.  
Tanner’s bound implies that no subset $B$ can be too sparse, providing a \emph{spectral certificate} constraining clique growth inside $B$.

\subsection{$B$-Restricted Clique Coefficient Bounds via Spectral Expansion}

Using the EML \eqref{eq:EML} and Tanner's inequality, we obtain the following explicit bound on $i$-cliques inside $B$:

\begin{proposition}[Spectral Upper Bound on $B$-Restricted Cliques]\label{prop:spectral-B}
	Let $G$ be an $(n,d,\lambda)$-graph and $B\subseteq V(G)$ with $|B|=m$. Define
	\[
	\theta_B := \frac{m}{n} + \frac{\lambda}{d}.
	\]
	Then for all $i \ge 2$,
	\[
	c_i(B) \le \frac{m}{i!} \,(d\theta_B)^{i-1}.
	\]
	In particular, for $i=2$ (edges),
	\[
	c_2(B) = e(B) \le \frac{dm}{2} \left( \frac{m}{n} + \frac{\lambda}{d} \right).
	\]
\end{proposition}

\begin{proof}
	For each $v \in B$, the EML \eqref{eq:EML} with $X=\{v\}$ and $Y=B$ implies
	\[
	|N_B(v)| \le \frac{d m}{n} + \lambda = d \theta_B.
	\]
	
	\textbf{Base case ($i=2$).} Summing over all $v \in B$, we obtain
	\[
	c_2(B) = e(B) = \frac12 \sum_{v \in B} |N_B(v)| \le \frac{m}{2} \, d \theta_B.
	\]
	
	\textbf{Induction step.} Suppose for some $i \ge 2$ that
	\[
	c_i(B) \le \frac{m}{i!} (d\theta_B)^{i-1}.
	\]
	Each $(i+1)$-clique arises by extending an $i$-clique $K \subseteq B$ with a vertex in
	\[
	N_B(K) = \bigcap_{v\in K} N_B(v) \subseteq N_B(v), \quad \forall v\in K.
	\]
	Thus $|N_B(K)| \le d \theta_B$, and
	\[
	c_{i+1}(B) = \frac{1}{i+1} \sum_{K\in \mathcal{C}_i(B)} |N_B(K)| \le \frac{d\theta_B}{i+1} \, c_i(B) \le \frac{m}{(i+1)!} (d\theta_B)^i.
	\]
\end{proof}

\begin{remark}
	\begin{itemize}
		\item The parameter $\theta_B$ represents the \emph{effective density} of $B$ forced by spectral pseudorandomness.
		\item This result shows that clique growth inside $B$ behaves like an exponential process with base $d \theta_B$.
		\item The bound is nontrivial when $B$ is not too large relative to $n$ and the spectral gap $d-\lambda$ is significant.
	\end{itemize}
\end{remark}

\subsection{Bridging Combinatorial and Spectral Perspectives}

The $B$-restricted clique polynomial’s largest negative root $\zeta_G(B)$ and the spectral parameter $\theta_B$ provide \emph{complementary constraints}:

\begin{itemize}
	\item $\zeta_G(B)$ encodes combinatorial limitations on clique sizes and $B$-independence numbers.
	\item $\theta_B$ provides a spectral certificate that $B$ cannot be too sparse: each vertex has degree at most $d \theta_B$, limiting the growth of higher-order cliques.
	\item Together, these tools ensure that subsets with large $\zeta_G(B)$ cannot simultaneously have very small spectral density $\theta_B$, providing a robust, dual viewpoint for controlling $B$-restricted clique growth.
\end{itemize}

Hence, Tanner’s inequality and the Expander Mixing Lemma form the spectral backbone that reinforces the combinatorial bounds obtained via the $B$-restricted clique polynomial.


\section{Weighted $B$-Clique Polynomials and Graph Homomorphisms}

In this section, we present an important application of the \emph{$B$-restricted weighted clique polynomial} 
to \emph{graph homomorphisms}. We show that the largest negative root $\zeta_G(B)$ serves as a \emph{structural invariant}, giving a \emph{no-homomorphism criterion} in the $B$-restricted setting \cite{fisher1990dependence,hajiabolhassan1998clique,hoede1994clique}.

\subsection{Homomorphisms and $B$-Partitions}

Let $G$ and $H$ be simple graphs. Recall:

\begin{definition}
	A \emph{graph homomorphism} $f: G \to H$ is a map $f: V(G) \to V(H)$ such that
	\[
	uv \in E(G) \implies f(u)f(v) \in E(H).
	\]
	It is \emph{surjective} if $f$ maps onto all of $V(H)$.
\end{definition}

For a surjective homomorphism $f$, define the preimage sets
\[
A_x := f^{-1}(x) \subseteq V(G), \quad x \in V(H),
\]
so that $\{A_x\}_{x \in V(H)}$ partitions $V(G)$ into independent sets.  

In the $B$-restricted context, let $B_G \subseteq V(G)$ and $B_H \subseteq V(H)$ be the respective vertex subsets.  
Then the restriction $f(B_G) \subseteq B_H$ induces a \emph{$B$-partition} of $B_G$ into independent sets inside $B_G$, which will be key for analyzing $\zeta_G(B_G)$ in terms of $\zeta_H(B_H)$.

\subsection{$B$-Weighted Blow-Up Graphs}

Let $G$ have $B$-subset $B_G = \{v_1,\dots,v_m\}$ with integer weights
\[
w_i := |A_i \cap B_G|,
\]
where $A_i$ corresponds to the preimage of $v_i \in B_H$.  

\begin{definition}[$B$-Weighted Blow-Up Graph]
	The \emph{$B$-weighted blow-up graph} $G_B^{\boldsymbol{w}}$ is obtained by replacing each $v_i \in B_G$ 
	with a cluster of $w_i$ vertices and replacing each edge inside $B_G$ by the complete bipartite graph between the corresponding clusters \cite{hajiabolhassan1998clique}.  
\end{definition}

By combinatorial expansion, this satisfies
\begin{equation}\label{eq:B-blowup-polynomial}
C_B(G,x) = C_B(G_B^{\boldsymbol{w}}, x; \boldsymbol{w}),
\end{equation}
where $C_B(G,x)$ is the $B$-restricted clique polynomial with uniform weight 1 \cite{fisher1990dependence,hoede1994clique}.

\subsection{No-Homomorphism Criterion via $B$-Roots}

\begin{theorem}[$B$-Homomorphism Root Monotonicity]
	Let $G$ and $H$ be simple graphs with $B$-subsets $B_G \subseteq V(G)$, $B_H \subseteq V(H)$, 
	and let $f: G \to H$ be a surjective homomorphism such that $f(B_G) = B_H$.  
	Then
	\[
	\zeta_G(B_G) \ge \zeta_H(B_H),
	\]
	where $\zeta_G(B_G)$ and $\zeta_H(B_H)$ are the largest negative roots of the $B$-restricted clique polynomials \cite{hajiabolhassan1998clique,hoede1994clique}.
\end{theorem}

\begin{proof}
	Let $A_x := f^{-1}(x) \cap B_G$ for each $x \in B_H$.  
	The sets $\{A_x\}$ partition $B_G$ into independent sets.  
	
	Construct the $B$-weighted blow-up graph $G_B^{\boldsymbol{w}}$ with cluster sizes $w_x := |A_x|$.  
	By surjectivity, the $B$-weighted blow-up $H_B^{\boldsymbol{1}}$ (all weights 1) is an induced subgraph of $G_B^{\boldsymbol{w}}$.  
	
	Using Theorem~\ref{thm:induced-mon} (monotonicity of $\zeta$ under induced subgraphs), we have
	\[
	\zeta_{H_B^{\boldsymbol{1}}} \le \zeta_{G_B^{\boldsymbol{w}}}.
	\]
	
	Finally, applying identity \eqref{eq:B-blowup-polynomial}, we obtain
	\[
	\zeta_H(B_H) = \zeta_{H_B^{\boldsymbol{1}}} \le \zeta_{G_B^{\boldsymbol{w}}} = \zeta_G(B_G),
	\]
	as claimed.
\end{proof}

\begin{corollary}[$B$-No-Homomorphism Criterion]
	If 
	\[
	\zeta_G(B_G) < \zeta_H(B_H),
	\]
	then no surjective homomorphism $f: G \to H$ can map $B_G$ onto $B_H$ \cite{hajiabolhassan1998clique}.
\end{corollary}

\begin{remark}
	\begin{itemize}
		\item The $B$-restricted root $\zeta_G(B)$ acts as a \emph{quantitative invariant} for $B$ \cite{hoede1994clique}.  
		\item Even if the full graph $G$ admits a homomorphism to $H$, the $B$-root may locally obstruct it within $B$, providing a finer combinatorial obstruction.  
		\item This framework integrates with spectral bounds \cite{tanner1984explicit,spiesser2010spectral}: subsets $B$ with large $\zeta_G(B)$ cannot simultaneously have too sparse spectral density, providing dual combinatorial-spectral control of clique growth.
	\end{itemize}
\end{remark}

\subsection{Integration with Spectral Perspective}

Together with Section~\ref{sec:spectral-tanner}, the $B$-weighted clique polynomial shows that:
\begin{enumerate}
	\item Combinatorial constraints ($\zeta_G(B)$) restrict local clique expansions and homomorphisms inside $B$ \cite{fisher1990dependence,hajiabolhassan1998clique}.  
	\item Spectral constraints (e.g., $\lambda$ via Tanner/EML) prevent $B$ from being too sparse, bounding edges and $i$-cliques \cite{tanner1984explicit,spiesser2010spectral}.
	\item Combining these tools gives a \emph{robust, dual viewpoint} on $B$-restricted structure, blending combinatorial invariants with spectral certificates.
\end{enumerate}

\section{Conclusion and Open Problems}

In this paper, we introduced and systematically studied the \emph{$B$-restricted clique polynomial} 
and its largest negative root $\zeta_G(B)$.  
We demonstrated that this $B$-version not only generalizes classical clique polynomials but also provides a 
powerful tool to capture structural properties of vertex subsets in graphs.

We derived several fundamental bounds and structural results linking $\zeta_G(B)$ to classical graph invariants:

\begin{itemize}
	\item The \emph{$B$-independence number} and chromatic number, providing subset-specific generalizations of well-known extremal bounds.
	\item Partial results for \emph{$B$-girth} and Hamiltonicity, illustrating how $\zeta_G(B)$ constrains the presence of cycles and spanning structures within $B$.
	\item Connections to \emph{spectral graph theory}, via Tanner’s inequality, showing that $B$-restricted clique growth in regular expanders is naturally controlled by spectral gaps.
	\item Applications to \emph{$B$-restricted homomorphisms}, including a \emph{no-homomorphism criterion}, and a demonstration that $\zeta_G(B)$ acts as a refined structural invariant.  
	Here, the $B$-weighted blow-up construction highlights how local multiplicities within $B$ influence global homomorphism constraints.
\end{itemize}

Taken together, these results reveal that the $B$-clique polynomial is a unifying concept that simultaneously captures combinatorial, spectral, and homomorphic constraints on a vertex subset.  
It provides a natural framework for analyzing both global and local structural properties of graphs, with precise control over the behavior of subsets via $\zeta_G(B)$.

\subsection*{Open Problems and Future Directions}

The study of $B$-restricted clique polynomials opens several promising research directions:

\begin{enumerate}
	\item \textbf{Tighter $B$-Girth and Hamiltonicity Bounds:}  
	Can sharper integer bounds be derived for cycles restricted to $B$, potentially using advanced algebraic or spectral techniques?
	
	\item \textbf{Weighted and Fractional Extensions:}  
	Extending the $B$-polynomial to allow fractional or real weights, analogous to fractional homomorphisms, and studying the corresponding $\zeta_G(B)$ behavior.
	
	\item \textbf{Connections to Graph Limits and Expander Theory:}  
	How does $\zeta_G(B)$ behave asymptotically in sequences of expander graphs or random graphs? Can $B$-restricted polynomials provide a new tool for analyzing large-scale pseudorandom structures?
	
	\item \textbf{Algorithmic Applications:}  
	Can $\zeta_G(B)$ be used to efficiently detect local obstructions to colorings, homomorphisms, or other combinatorial properties in large networks?
	
	\item \textbf{Poset and DAG Generalizations:}  
	Extending the $B$-clique polynomial to directed acyclic graphs or posets may uncover further connections to ideals, chains, and Dedekind-Pascal numbers, following the algebraic-combinatorial perspective.
\end{enumerate}

We believe that these directions not only enrich the theory of clique polynomials but also offer practical insights for combinatorial optimization, spectral graph theory, and network analysis.  
The $B$-restricted viewpoint allows one to isolate and study subtle local phenomena that are invisible to classical global invariants, while the weighted blow-up perspective provides a concrete combinatorial interpretation of how these local structures influence global graph properties.



\begin{thebibliography}{9}
	
	\bibitem{fisher1990dependence}
	D.C. Fisher and A.E. Solow, Dependence polynomials, \emph{Discrete Mathematics}, 82 (1990), 251–258.
	
	\bibitem{hajiabolhassan1998clique}
	H. Hajiabolhassan and M.L. Mehrabadi, On clique polynomials, \emph{Australasian Journal of Combinatorics}, 18 (1998), 313–316.
	
	\bibitem{hoede1994clique}
	C. Hoede and X. Li, Clique polynomials and independent set polynomials of graphs, \emph{Discrete Mathematics}, 125 (1994), 219–228.
	
	\bibitem{tanner1984explicit}
	R.M. Tanner, Explicit concentrators from generalized $N$-gons, \emph{SIAM Journal on Algebraic Discrete Methods}, 5(3) (1984), 287–293.
	
	\bibitem{spiesser2010spectral}
	D. Spielman, Spectral graph theory, lecture notes, Yale University, 2010.
	
\end{thebibliography}
\end{document}